\documentclass[11pt]{article}

\usepackage{amsmath}
\usepackage{amssymb}
\usepackage{fullpage}

\newtheorem{thm}{Theorem}

\newtheorem{lemma}[thm]{Lemma}

\newtheorem{prop}[thm]{Proposition}

\newenvironment{proof} { \emph{Proof.} } { {\rule{2mm}{2mm}}\\}

\newcommand{\bea}{\begin{eqnarray*}}
\newcommand{\eea}{\end{eqnarray*}}

\newcommand{\themap}{{\Xi}}

\newcommand{\e}{\epsilon}

\newcommand{\R}{\mathbb{R}}

\newcommand{\B}{\mathcal{B}}

\newcommand{\G}{{\mathcal{G}}}
\newcommand{\N}{{\mathcal{N}}}

\newcommand{\g}{{\gamma}}
\newcommand{\x}{{\mathbf{x}}}
\newcommand{\y}{{\mathbf{y}}}
\newcommand{\z}{{\mathbf{z}}}

\newcommand{\bm}{\begin{pmatrix}}
\newcommand{\fm}{\end{pmatrix}}

\begin{document}
\title{Bilipschitz maps of boundaries of certain negatively curved homogeneous spaces}
\author{Tullia Dymarz and Irine Peng}
\maketitle
\abstract{
In this paper we study certain groups of bilipschitz maps of the boundary minus a point of a negatively curved space of
the form $\mathbb{R} \ltimes_{M} \mathbb{R}^{n}$, where $M$ is a matrix whose eigenvalues
all lie outside of the unit circle.  The case where $M$ is diagonal was previously studied
by Dymarz in \cite{Dy}.  As an application, combined with work of Eskin-Fisher-Whyte and Peng, we provide the last steps in the proof of quasi-isometric rigidity for a class
of lattices in solvable Lie groups.}

\section{Introduction}

In \cite{H} Heintze characterized the class of connected, negatively curved homogeneous
spaces as those solvable Lie groups of the form $\mathbb{R} \ltimes N$, where N is a
nilpotent Lie group, and where the eigenvalues of the action of $\mathbb{R}$ on $N$ all lie strictly outside the unit circle.  If $X$ is such a space then the visual
boundary $\partial X$ is defined to be the set of equivalence classes of geodesic rays.
This visual boundary can be identified with the one-point compactification of $N$, and it
can be be equipped with a metric, the visual metric, by fixing a
reference point in $X$ and examining the Gromov product between the reference point and
any pairs of points on the boundary. There is another family of metrics $\{ d_{a,
\mathcal{H}} \}$, called Euclid-Cygan metrics, that can be defined on $\partial X \setminus \{
a\}$, for any point $a \in \partial X$ and any horosphere $\mathcal{H}$ centered at $a$. In this
metric, isometries of $X$ fixing $a$ act by homotheties.  There is an explicit
relationship between the restriction of the visual metric to $\partial X \setminus \{ a \}$ and a
Cygan metric given by Paulin \cite{HP}. In particular
the two metrics are quasiconformal. \\

In this paper we study bilipschitz maps with respect to a Euclid-Cygan metric $D_{M}$ on the
boundary of a negatively curved space $G_{M}=\mathbb{R} \ltimes_{M} \mathbb{R}^{n}$,
where $\R$ acts on $\R^n$ by a one parameter subgroup $M^t \subseteq GL(n,\R)$ such that
$M$ is a matrix whose eigenvalues all have norm greater than one. We call such maps $Bilip_{D_M}$ maps.
We show that with respect to a certain (partially) ordered basis $\mathcal{B}$ (see Section \ref{basissection})  all $Bilip_{D_M}$ maps have the following upper triangular form. 

\begin{prop} Let $F$ be a $Bilip_{D_M}$ map of $\R^n$. Then for $u \in \R^n$
$$F(u) =(f_1(\x_1, \ldots, \x_r),\ldots, f_r(\x_r))$$
where $(\x_1, \ldots, \x_r)$ are the coordinates of $u$ with respect to $\mathcal{B}$. Furthermore, each $f_i$ is bilipschitz with respect to $\x_i$ and continuous in the remaining coordinates.  
\end{prop}

Next we consider
groups of uniform $QSim_{D_M}$ maps, maps which are compositions of homotheties
($Sim_{D_M}$ maps) and $Bilip_{D_M}$ maps all with a
uniform bilipschitz constant. Under certain conditions we are able to conjugate such a
group into the group of almost homotheties ($ASim_{D_M}$ maps). See Section
\ref{two} for a precise definition. We prove the following theorem:

\begin{thm}\label{tukia} Let $\G$ be a cocompactly acting uniform group of $QSim_{D_M}$ maps of
$(\R^n, D_M)$ where $M$ is a matrix with all eigenvalues of norm greater than
one. Then there exists a $QSim_{D_M}$ map $F$ such that $$F \G F^{-1} \subset
ASim_{D_M}(\R^n).$$ \end{thm}


This theorem generalizes the work of Dymarz who considers the same problem in the setting where $M$ is
diagonalizable. Theorem \ref{tukia} fills in the missing ingredient in the proof of the following theorem:

\begin{thm}\cite{EFW}\label{efw} Let $\Gamma$ be a finitely generated group quasi-isometric to $\R\ltimes_M \R^n$ where $M$ is a
matrix with $\det{M}=1$ and whose eigenvalues all lie off of the unit circle. Then $\Gamma$ is virtually a lattice in a
solvable Lie group of the form $\R \ltimes_{M^\alpha} \R^n$ for some $\alpha \in \R_+$. \end{thm}

\subsection{Outline}  In Section \ref{two} we prove the main results
concerning  the structure of $Bilip_{D_M}$ maps. We show that such a map preserves a flag of foliations determined by the generalized eigenspaces of $M$. This is where the non-diagonal case differs most from the diagonal case. In Section \ref{tukiasection} we prove Theorem \ref{tukia} and in Section \ref{endgame} we fill in the gaps of the proof of Theorem \ref{efw}.

\section{The space $G_{M}$ and $\partial G_M$}\label{two}
Given a $n \times n$ matrix $A$ with positive real eigenvalues, the eigenvalues of the
matrix $M=e^{A}$ are all greater than $1$.  Conversely, any matrix with eigenvalues
greater than one can be written in the form $e^A$ where $A$ has positive
eigenvalues. 
Furthermore, up to a compact factor and squaring if necessary, any matrix with eigenvalues outside the unit circle can be identified with a matrix with eigenvalues greater than one.\footnote{maybe leave this for a later point when talking about quasi-isometries}
Under this assumption, the semidirect product $G_M=\mathbb{R} \ltimes_{M}
\mathbb{R}^{n}$, where the $\R$ action is given by $t \mapsto e^{tA}$, is a connected and
negatively curved homogeneous space.  \smallskip

We can equip $G_{M}$ with a path metric $d_R$ induced from a left invariant Riemannian metric but to work with this metric explicitly would be cumbersome.
Instead, we chose to work with a metric $d_L$ which is bilipschitz equivalent to $d_R$ and much easier to describe.
To define $d_L$ we first assign a metric to each height level set as in \cite{FM3}:
$$d_{t,M}(p,q):= \| e^{-tA} (p-q) \|$$
where $\| \cdot \|$ is the coordinate-wise max norm.  For two points $(t,p),(t',q)\in G_M$ let $t_0$ be the smallest height at which
$d_{t_0,M}(p,q)\leq 1$.
If $t_0 \geq t,t'$ then we set
 \[  d_L((t,p), (t',q)) := |t-t_0| + |t_0-t'| +1\]
 otherwise (assuming that $t>t'$) we set
 \[  d_L((t,p), (t',q)) := (t-t') + \| e^{- tA} (p-q) \|. \]





\noindent As in \cite{Dy} we identify $\partial G_M \setminus \{\infty\} \simeq \R^n$
and define the boundary metric $D_M$ by setting for $p,q \in \partial G_M \setminus \{\infty\} $
$$D_M(p,q)=e^{t_0}$$
where $t_0$ is the smallest height at which $\| e^{-t_0A} (p-q) \| \leq 1$.
In the following section we derive a coordinate based expression for $D_M$.\\

\noindent{\bf Remark.} To compare $D_M$ with the
Euclid-Cygan metric from \cite{HP}
let $\mathcal{H}$ be the horosphere centered at $\infty$ defined by $t=0$.  Let $p_t, q_t$ be vertical geodesics based at $p, q \in \partial G_M$ respectively, parametrized so that $p_0, q_0$ both lie on $\mathcal{H}$ and so that $t$ corresponds to height.
Then the Euclid-Cygan metric is given by
%
\bea
d_{\infty, \mathcal{H}}(p,q)  &=& \lim_{t \rightarrow -\infty} e^{-\frac{1}{2}(-2t - d_{L}(p_{t}, q_{t}))}\\
&=& \lim_{t \rightarrow -\infty} e^{\frac{1}{2}(2t +2(t_0-t) +1)}\\
&=& Ce^{t_0}.
\eea



%

\subsection{The $D_{M}$ metric in coordinates}\label{basissection}
In this subsection, given $M=e^A$ where $A$ is a $n \times n$ matrix with positive real eigenvalues as before,
we work out the expression of $D_{M}$ in coordinates with respect to
the basis in which the matrix $A$ appears in its Jordan canonical form.

Let
$V_{\alpha}$ be the generalized eigenspace corresponding to the eigenvalue $\alpha$, and
let $\{ v^{\alpha}_{i} \}_{i}$ be generators of Jordan chains in $V_{\alpha}$. Each
$v^{\alpha}_{i}$ is associated with a number $n_{i}$ such that $(A-\alpha I)^{n_{i}}
v^{\alpha}_{i}=0$, and $(A-\alpha I)^{n_{i}-1} v^{\alpha}_{i} \not= 0$. Then $V_{\alpha}$
has a basis of the form
\[ \B_{\alpha} = \bigcup_{i} \{(A-\alpha I )^{n_{i}-1} v^{\alpha}_{i}, \cdots, (A-\alpha I)^{1} v^{\alpha}_{i},
v^{\alpha}_{i}\}. \]
Let \[ \B_{\alpha, \ell}=\B_{\alpha} \cap \left( \mbox{ker}(A-\alpha I)^{\ell}
\backslash \mbox{ker} (A-\alpha I)^{\ell - 1} \right). \]
\noindent Then a basis of $\mathbb{R}^{n}$ can be given by
\[ \B= \bigcup_{\alpha} \bigcup_{\ell} \B_{\alpha, \ell}. \]
We now turn $\B$ into an ordered basis by fixing an order on each
$\B_{\alpha, \ell}$ and by declaring elements of $\B_{\alpha,\ell}$  to 
take precedence over elements in $\B_{\beta, \iota}$ if either $\alpha >
\beta$, or $\alpha=\beta$ and $\ell > \iota$.
This ordering on $\B$ produces a foliation $\{ U_{\alpha,\ell}  \}$ of $\R^n$ where
\[ U_{\alpha, \ell} = \mbox{span } \{  \bigcup_{(\beta,j) < (\alpha, \ell) } \B_{\beta,j}  \}. \]

\noindent{}We can coordinatize each element $u \in \mathbb{R}^{n}$ with respect to
$\B$ as $(\mathbf{x}_{\alpha,1}, \mathbf{x}_{\alpha,2},
\cdots \mathbf{x}_{\beta,1}, \mathbf{x}_{\beta, 2} \cdots )$ where $\alpha < \beta$, and
$\mathbf{x}_{\alpha,j}$ is a vector of coefficients with respect to elements of
$\B_{\alpha, j}$ written in the ascending order.\\


In these coordinates, write $p-q=(\ldots,\Delta \x_{\alpha,j},\ldots)$. Then we can express
$$D_M(p,q)=e^{t_0}$$ where $t_0$ is the smallest value of $t$ that satisfies the inequality
\begin{equation}\label{DMdefn} 1 \geq
\max_{\alpha, \ell,j} e^{-\alpha t} \left|\sum_{i=j}^{\ell}(-1)^{i}\frac{t^{i-j}}{(i-j)!}
(\Delta \x_{\alpha,i}) \right|.
\end{equation}
In particular, for some $j_0,\ell_0$ and $\alpha_0$ we have
$$D_M(p,q)= \left| \sum_{i=j_0}^{\ell_0}(-1)^{i}\frac{t_0^{i-j_0}}{(i-j_0)!}(\Delta \x_{\alpha_0,i})\right|^{1/\alpha_0}.$$


%
%
\subsection{Map definitions}
\noindent{}{\bf Bilipschitz maps.} We say a map $F: \mathbb{R}^{n} \rightarrow
\mathbb{R}^{n}$ is \emph{bilipschitz with respect to $D_M$} if there is a constant $K$
such that \[ \frac{1}{K} D_{M}(p,q) \leq D_{M}(F(p), F(q)) \leq K D_{M}(p,q)  \mbox{ for
all  } p, q \in \mathbb{R}^{n}. \] \noindent We call such a map a $Bilip_{D_M}$ map.\\

\noindent{}{\bf Similarity.} A map $F: \mathbb{R}^{n} \rightarrow \mathbb{R}^{n}$ is a $Sim_{D_M}$ map or a 
\emph{similarity with respect to $D_{M}$} if there is a constant $c$ such that \[
D_{M}(p,q)=e^{c} D_{M}(F(p),F(q)) , \mbox{ for all } p,q \in \mathbb{R}^{n}. \]\smallskip

\noindent{}{\bf Quasi-similarity.} A $QSim_{D_M}$ map is a $Bilip_{{D}_M}$ map composed with a
$Sim_{D_M}$ map.
 In particular a $QSim_{D_M}$ map is again $Bilip_{D_M}$ but this distinction is necessary because when we consider \emph{uniform} groups of $QSim_{D_M}$ maps we will mean uniform bilipschitz constants up to composing with a similarity. \\

\noindent{}{\bf Almost similarity.} \emph{Almost similarities} or $ASim_{D_M}$ maps form a restricted subset of all $QSim_{D_M}$ maps. In the next section we show that any $Bilip_{D_M}$ preserves a certain flag of foliations defined by the generalized eigenspaces of $M$. An $ASim_{D_M}$ map is one which, in addition to preserving these foliations, induces a similarity map along each leaf of this foliation such that all of the dilation and rotation constants are compatible with $M$. For an explicit definition see Section \ref{tukiasection}.

\subsection{Properties of  $Bilip_{D_{M}}$ maps.}\label{bilipDM}

In this section we examine the structure of $Bilip_{D_{M}}$ maps.  Our main results are
Proposition \ref{myfoliationlemma} and Lemma \ref{XitheMap} which show that a $Bilip_{D_M}$ map must preserve a certain flag of foliations defined by eigenvalues and Jordan block filtrations of $M$. Furthermore, these results give growth conditions that a $Bilip_{{D}_M}$ must necessarily satisfy along leaves of these foliations.\smallskip

Define a function $\eta_{\alpha,j}: \R \to \R$ by $$ \eta_{\alpha,j}(\omega)=\frac{j!\ \omega^{\alpha}}{|\ln{\omega}|^{j}}. $$
For two points $p,q \in \R^n$ let $$\triangle_{\eta_{\alpha,j}}(p,q) = \liminf_{ k \to
\infty} \sum \eta_{\alpha,j}(D_M(p_i,p_{i+1}))$$ where the $\liminf$ is taken over all finite
sequences $\{ p_i\}_{i=0}^m$ with  $p=p_0$, $q=p_m$ and $D_M(p_i,p_{i+1}) = 1/k$. Note that
$\triangle_{\eta_{\alpha,j}}$ can be thought of as the ``length'' metric associated to the
``metric'' $\eta_{\alpha,j}\circ D_M$.

\begin{lemma}\label{bilipTriangle}
If $F:\R^n \to \R^n$ is a $K$-$Bilip_{D_M}$ map then
\[ \frac{1}{K^{\alpha}} \triangle_{\alpha,j}(p,q) \leq \triangle_{\alpha,j}(F(p),F(q)) \leq K^{\alpha}\triangle_{\alpha,j}(p,q) \]
\end{lemma}
\begin{proof}
Let $\{ p_i \}_{0}^{m}$ be a sequence of points such that $p_{0}=p$ and $p_{m}=q$. Then,
since $F$ is $K$-bilipschitz, we have
\[ \frac{1}{K} D_{M}(p_{i-1},p_{i}) \leq D_{M}(F(p_{i-1}),F(p_{i})) \leq K
D_{M}(p_{i-1},p_{i})  \mbox{ for all  }   i.  \]
Since $\eta_{\alpha,j}$ is monotone, it follows that
\[  \sum_{j=1}^{m} \eta_{\alpha,j}\left(\frac{1}{K}D_M(p_{i-1},p_i)\right)\leq \sum_{j=1}^{m}
\eta_{\alpha,j}(D_M(F(p_{i-1}),F(p_i)))\leq  \sum_{j=1}^{m}  \eta_{\alpha,j}(KD_M(p_{i-1},p_i)). \]
We can estimate \bea
\eta_{\alpha,j}( K D_M(p_{i-1},p_i))&=&\frac{j!\ K^{\alpha} D_M(p_{i-1},p_i)^{\alpha}}{ |\ln K + \ln D_M(p_{i-1},p_i)|^{j}}\\
&=& \frac{j!\ K^{\alpha} D_M(p_{i-1},p_i)^{\alpha}}{|\ln D_M(p_{i-1},p_i)|^{j}} \cdot
\frac{|\ln D_M(p_{i-1},p_i)|^{j}}{|\ln K + \ln D_M(p_{i-1},p_i)|^{j}}. \eea
Since
\[ \lim_{D_{M}(p_{i-1},p_{i}) \rightarrow 0} \frac{|\ln D_M(p_{i-1},p_i)|^{j}}{|\ln K + \ln D_M(p_{i-1},p_i)|^{j}}
\to 1 \] \noindent the claim follows from the definition of $\triangle_{\alpha,j}$. \end{proof}

\begin{lemma}\label{trianglelemma}
Let $p,q \in \partial G_{M}$ be points such that $\x_{\beta,j}$ is the largest coordinate in which $p$
differs from $q$.
Then
\begin{itemize}
\item $(\beta,j) > (\alpha, \ell)$ if and only if
    $\triangle_{\eta_{\alpha,\ell}}(p,q)= \infty$,

\item $(\beta,j) =(\alpha, \ell)$ if and only if $\triangle_{\eta_{\alpha,\ell}}(p,q)= |\x_{\alpha,\ell}|$,

\item $(\beta,j) < (\alpha, \ell)$ if and only if
    $\triangle_{\eta_{\alpha,\ell}}(p,q)= 0$. \end{itemize}
\end{lemma}


\begin{proof}
First we will find an upper bound on $\triangle_{\eta_{\alpha,\ell}}(p,q)$. Suppose $p-q=\x_{\beta,j}$.
(i.e. $p$ and $q$ differ only in the $\x_{\beta,j}$ coordinate.)
Then
\begin{equation}\label{onevarDMeq}D_M(p,q)=\max \left\{ \left|\frac{t^j}{j!}  \x_{\beta,j}\right|^{1/\beta},\left|\frac{t^{j-1}}{(j-1)!} \x_{\beta,j}\right|^{1/\beta},\ldots, \left|\x_{\beta,j}\right|^{1/\beta}\right\}.\end{equation}
Let $\{p_i\}$ be a sequence joining $p$ and $q$ such that the subsequent terms differ only in the $\x_{\beta,j}$ coordinate. Write $p_i-p_{i+1}=\x_{\beta,j}^i$.
If we examine Equation \ref{onevarDMeq} when $D_M(p_i,p_{i+1})=1/k$ then we notice that the first term $ |t^j/j!\ \x^i_{\beta,j}|^{1/\beta}$ dominates 
as long as $k$ is large enough. 
Then
$$|\x_{\beta,j}^i|=\frac{j!}{k^\beta |\ln{k}|^j }$$
so that we need
$$\frac{|\x_{\beta,j}|}{|\x_{\beta,j}^i|}= k^\beta |\ln{k}|^j  |\x_{\beta,j}|(1/j!)$$
points in our sequence.
Now consider
\bea
\sum_{i=0}^{ k^\beta |\ln{k}|^j |\x_{\beta,j}| (1/j!)} \eta_{\alpha,\ell}(D_M(p_i,p_{i+1}))
& =& \sum_{i=0}^{ k^\beta |\ln{k}|^j |\x_{\beta,j}| (1/j!)} \eta_{\alpha,\ell}(1/k)
=\sum_{i=0}^{ k^\beta |\ln{k}|^j |\x_{\beta,j}| (1/j!)}  \frac{\ell !}{k^\alpha  |\ln{k}|^\ell }\\
& =& \frac{ k^\beta |\ln{k}|^j |\x_{\beta,j}| \ell ! }{k^\alpha |\ln{k}|^\ell j ! \ }.
\eea
This quantity above is an upper bound for $\triangle_{\eta_{\alpha,\ell}}(p,q)$.
Note that if
\begin{itemize}
\item $\alpha=\beta$
\begin{enumerate}
\item[a)] $\ell=j$ then  $\triangle_{\eta_{\alpha,\ell}}(p,q)\leq|\x_{\beta,j}|$,
\item[b)] $\ell >j$ then $(\ell ! /j!)|\ln{k}|^{j-\ell}|\x_{\beta,j}| \to 0$ as $1/k \rightarrow 0$, so $\triangle_{\eta_{\alpha,\ell}}(p,q)=0$,
\item[c)] $\ell <j$ then $(\ell !/j!)|\ln{k}|^{j-\ell}|\x_{\beta,j}|\to \infty$ as $1/k \rightarrow
0$,
\end{enumerate}
\item $\alpha > \beta$ then $k^{\alpha-\beta}$ dominates $|\ln{k}|^{j-\ell}$ no matter what $j$ and $\ell$ are so $\triangle_{\eta_{\alpha,\ell}}(p,q)=0$,
\item $\alpha < \beta$ then  $k^{\beta- \alpha} |\ln{k}|^{j-l}|\x_{\beta,j}| (\ell ! /j!)\to \infty$.
\end{itemize}
Now if $p$ and $q$ differ in more than one coordinate then we can treat one coordinate at a time and concatenate all resulting sequences to get the desired upper bounds.

To find a lower bound on $\triangle_{\eta_{\alpha,\ell}}(p,q)$, we again look at various
cases. By the above estimates we already know that if the largest coordinate in which $p$
and $q$ differ is $(\beta, j)$ where $\alpha > \beta$  or where $\alpha=\beta$ and  $\ell
> j$ then $\triangle_{\eta_{\alpha,\ell}}(p,q)=0$. To consider the other cases, pick any sequence $\{ p_i\}$ with
$D_M(p_i,p_{i+1})=1/k$. First, if $\beta = \alpha$ then by Equation \ref{DMdefn}
$$\eta_{\alpha,\ell}(D_M(p_i,p_{i+1})) \geq \frac{\ell!}{|\ln{k}|^\ell} \left| \sum_{\iota=0}^j  (-1)^\iota \frac{|\ln k|^\iota}{ \iota !} \x^i_{\beta, \iota} \right|.$$
By applying the triangle inequality we get
\bea
\sum_{i=0}^{m} \eta_{\alpha,\ell}(D_M(p_i,p_{i+1}))&\geq&
\sum_{i=0}^{m} \frac{\ell!}{|\ln{k}|^\ell} \left| \sum_{\iota=0}^j  (-1)^\iota \frac{|\ln k|^\iota}{ \iota !} \x^i_{\beta, \iota} \right|\\
&\geq& \frac{\ell!}{|\ln{k}|^\ell} \left |\sum_{i=0}^{m}  \sum_{\iota=0}^j  (-1)^\iota \frac{|\ln k|^\iota}{ \iota !} \x^i_{\beta, \iota} \right|\\
&\geq& \frac{\ell!}{|\ln{k}|^\ell} \left |\sum_{\iota=0}^j  (-1)^\iota \frac{|\ln k|^\iota}{ \iota !}\sum_{i=0}^{m}  \x^i_{\beta, \iota} \right|\\
&\geq& \frac{\ell!}{|\ln{k}|^\ell} \left |\sum_{\iota=0}^j  (-1)^\iota \frac{|\ln k|^\iota}{ \iota !} \x_{\beta, \iota} \right|.
\eea To get the last inequality we use that $\x_{\beta,j}=\sum_{i=0}^m\x^i_{\beta,j}$.
From these inequalities we can deduce that as $1/k \rightarrow 0$,
\begin{itemize}
\item if $j=\ell$ then $\triangle_{\eta_{\alpha,\ell}}(p,q)\geq  |\x_{\alpha,\ell}|$,
\item if $j>\ell$ then $\triangle_{\eta_{\alpha,\ell}}(p,q)=\infty$.
\end{itemize}

The last case we need to consider is when $\beta> \alpha$. Since we have considered all
other cases we can assume without loss of generality that $p$ and $q$ differ only in
coordinates with exponent greater than $\alpha$. Since $D_M(p_i,p_{i+1})=1/k$ then by
Equation \ref{DMdefn} we have
$$\sum_{i=0}^{m} \eta_{\alpha,\ell}(D_M(p_i,p_{i+1})) \geq \frac{\ell!}{|\ln{k}|^\ell} \sum_{i=0}^m\frac{1}{k^{\alpha}}$$
where $m=k^\beta P(|\ln{k}|)|\x_{\beta,j}|$ for some polynomial $P$. Therefore $\triangle_{\eta_{\alpha,\ell}}(p,q)=\infty$ when $\beta> \alpha$.


\end{proof}
 %
 %

\begin{prop}\label{myfoliationlemma}
Let $F((\x_{\alpha,\ell}))=((\y_{\alpha,\ell}))$ be a $Bilip_{D_M}$ map.  Then for all $(\alpha,\ell)$

\begin{enumerate}
\item $F$ preserves foliations by
$$U_{\alpha,\ell}=span\{\bigcup_{(\beta,j) < (\alpha, \ell)} \B_{\beta,j}\}.$$

\item The image coordinates $\y_{\alpha,\ell}$ are given by
 $$ \y_{\alpha,\ell} = f_{\alpha, \ell} (( \x_{\beta,j})_{(\beta,j) \geq (\alpha, \ell)})
$$ \noindent  where $f_{\alpha,\ell}$, when considered as a function of
$\x_{\alpha,\ell}$, is bilipschitz with respect to the usual Euclidean metric. \end{enumerate} \end{prop}

\begin{proof}
The first claim follows from Lemmas \ref{bilipTriangle} and \ref{trianglelemma}.  For
the second claim, let $(\alpha,\ell)+1$ be the smallest index bigger than
$(\alpha,\ell)$.  Take two points $p,q$ belonging to the same $U_{(\alpha, \ell)+1}$
coset, but differing in the $(\alpha,\ell)$-th coordinate.  By the first claim, we
know that $F(p)$ and $F(q)$ belong to the same $U_{(\alpha,\ell)+1}$ coset as well.
By Lemma \ref{trianglelemma}
$$\triangle_{\alpha,\ell}(p,q)= |\x_{\alpha,\ell}| \textrm{ and } \triangle_{\alpha,\ell}(F(p),F(q))=|\y_{\alpha,\ell}|,$$  and so the result follows from Lemma
\ref{bilipTriangle}.  \end{proof}

\begin{lemma}\label{XitheMap}
Suppose $p$ and $q$ differ only in the $\x_{\beta,j}$ coordinates. Then for all $(\alpha, \ell) < (\beta,j)$,
$$ |f_{\alpha,\ell} (p) - f_{\alpha,\ell} (q)|< \themap( |\x_{\beta,j}| ) $$ where
$f_{\alpha,\ell}(p)$ is the $(\alpha,\ell)$-th coordinate of $F(p)$,  $\x_{\beta,j}=p-q$
and $\themap$ is a nondecreasing function that goes to zero as its input approaches zero.
\end{lemma}
\begin{proof}
To aide with the
exposition we will define
$$\Upsilon( \x_{\beta,j}):=\max \{ |t^j/j!
\x_{\beta,j}|^{1/\beta},|t^{j-1}/(j-1)! \x_{\beta,j}|^{1/\beta},\ldots,
|\x_{\beta,j}|^{1/\beta}\}$$
where $t=\ln(D_M(p,q))$.\\

\noindent{\bf Claim:}  $\Upsilon(\x_{\beta,j})$ is non decreasing and goes to
$0$ as $|\x_{\beta,j}| \to 0 $.\\

\noindent{}The claim can be verified by a simple calculation.
Note that when $p$ and $q$ differ only in the $\x_{\beta,j}$ coordinate then $D_M(p,q)=
\Upsilon( \x_{\beta,j})$ and so, since $F$ is $Bilip_{D_M}$,
$$D_M(F(p),F(q)) \leq K \Upsilon( \x_{\beta,j}).$$
However, we also have
$$ D_M(F(p),F(q))^\alpha  \geq | f_{\alpha, \ell} (p)- f_{\alpha,
\ell} (q)  + \sum_{\ell < k \leq \ell_\alpha } (-1)^{(k-\ell)}
\frac{t'^{k-\ell}}{(k-\ell)!} ( f_{\alpha, k} (p) - f_{\alpha, k} (q))|$$
where $t- \e \leq t' \leq t + \e$ and $\ell_\alpha$ is the size of the largest Jordan
block with eigenvalue $\alpha$ (in other words the largest nilpotence degree
associated to the eigenvalue $\alpha$). Combining the above two inequalities we get
\begin{equation}\label{bigeqn}
| f_{\alpha, \ell} (p)- f_{\alpha, \ell} (q) | \leq K^\alpha
\Upsilon( \x_{\beta,j})^\alpha +  \sum_{\ell < k \leq \ell_\alpha }
|\frac{t'^{k-\ell}}{(k-\ell)!} ( f_{\alpha, k} (p) - f_{\alpha, k} (q))|.
\end{equation}
Note that when $ \ell = \ell_\alpha$ then
$$| f_{\alpha, \ell} (p)- f_{\alpha, \ell} (q) | \leq K^\alpha
\Upsilon( \x_{\beta,j})^\alpha. $$
By the claim above, when $p$ and $q$ differ only in the $\x_{\beta,j}$ coordinate and $\ell=\ell_\alpha$, we can take $$\themap( \x_{\beta,j}) =  K^\alpha
\Upsilon( \x_{\beta,j})^\alpha.$$

At this point note that $ \Upsilon(
\x_{\beta,j})=D_M(p,q)=e^t$ so that if $\Upsilon( \x_{\beta,j}) \to 0$ then $|t|^s \Upsilon( \x_{\beta,j}) \to 0$ for any $s$.
Using Equation \ref{bigeqn}, we can proceed inductively to show that we can find a function $\themap(\x_{\beta, j})$ such that
$$| f_{\alpha, \ell} (p)- f_{\alpha, \ell} (q) | \leq \themap(\x_{\beta, j})$$
with the property that $|t|^s \themap(\x_{\beta, j}) \to 0$ as $|\x_{\beta, j}| \to 0$ for any power of $|t|$. We
assume that for $k> \ell$
$$|f_{\alpha, k} (p) - f_{\alpha, k} (q))| < \themap'(\x_{\beta, j}).$$
where $\themap'$ is a map with the above desired properties.
By inequality \ref{bigeqn} we have that \bea | f_{\alpha, \ell} (p)- f_{\alpha, \ell}
(q) | & \leq & K^\alpha \Upsilon( \x_{\beta,j})^\alpha +  \sum_{\ell < k \leq
\ell_\alpha } |\frac{t'^{k-\ell}}{(k-l)!} ( f_{\alpha, k} (p) - f_{\alpha, k}
(q))|\\
& \leq & K^\alpha  \Upsilon( \x_{\beta,j})^\alpha +  \sum_{\ell < k \leq \ell_\alpha
} \frac{(|t|+ \e)^{k-\ell}}{(k-l)!} \themap'(\x_{\beta,j}). \eea
Therefore setting
$$\themap(\x_{\beta,j}) = K^\alpha  \Upsilon( \x_{\beta,j})^\alpha +
\sum_{\ell < k \leq \ell_\alpha } \frac{(|t|+ \e)^{k-\ell}}{(k-l)!}
\themap'(\x_{\beta,j})$$
gives us the desired bounding map.
\end{proof}

\subsection{Relating $(\R^n, D_M)$ to $\partial G_M$.}\label{boundary}
To understand the connection between $D_M$ and $G_M$ we need the notions \emph{height-respecting} isometries and quasi-isometries of $G_M$.
We define a \emph{height} function $h:G_M \to G_M$ by setting $h(t,p)=t$ for each $(t,p) \in G_M$. We say $f$ is a \emph{height-respecting} isometry (resp. quasi-isometry) if it permutes level sets of the height function (resp. permutes level sets up to a bounded distance). By \cite{FM3} this condition actually ensures that $f$ induces (resp. induces up to a bounded distance) a translation map on the height factor. In this section
we show that a height-respecting quasi-isometry of $G_M$ induces a $QSim_{D_M}$ maps of $\partial G_M$. This argument also appears in \cite{Dy} for the diagonalizable case.

\begin{lemma}\label{hrlemma} A height-respecting quasi-isometry (resp. isometry) of $G_M$ induces
a $QSim_{D_M}$ map (resp. $Sim_{D_M}$ map) of $\partial G_M \simeq \R^n$. \end{lemma}
\begin{proof}
Since a height-respecting quasi-isometry $\varphi$ of $G_{M}$ necessarily sends vertical
geodesics to (bounded neighborhoods) of vertical geodesics (see \cite{EFW}), $\varphi$ induces a well defined map $F$ on $\partial G_M \simeq \R$.
It follows that if $T$
(resp. $T'$) is the height at which the vertical geodesics emanating from $p$ and $q$
(resp. $F(p)$ and $F(q)$) are at distance one apart then
\[ \| M^{-T} (p-q) \| =O(1)  \mbox{  if and only if  } \| M^{-T'} \left( F(p)-F(q) \right) \| = O(1). \]
Now a height-respecting quasi-isometry of $G_M$ is just the composition of an isometry and a
quasi-isometry that fixes the identity and sends height level sets to bounded neighborhoods of height level sets. It
follows that there is a constant $a$ (depending on the isometry) and some $\epsilon$
(depending on the additive constant of the quasi-isometry) such that
\[ T+a -\epsilon \leq T' \leq T+a + \epsilon. \]

\noindent This implies that

\[ e^{-\epsilon} e^{a} e^{T}  \leq  e^{T'} \leq  e^{\epsilon} e^{a} e^{T} \]

\noindent and so \[ e^{-\epsilon} e^{a} D_{M}(p,q) \leq D_{M}(F(p),F(q)) \leq
e^{\epsilon} e^{a} D_{M}(p,q) \] \noindent by the definition of $D_{M}$.
\end{proof}

\section{Proof of Theorem \ref{tukia}  }\label{tukiasection}
For the most part, the proof of Theorem \ref{tukia} in the case when $M$ is not diagonalizable is very similar to the proof when $M$ is diagonalizable.
In this section, we will give an outline of the proof of Theorem \ref{tukia} and only fill in the details when the proofs of the two cases differ.
For a complete proof of Theorem \ref{tukia} see \cite{Dy}.

Given a matrix $M$ we coordinatize $V \simeq \R^n$ with respect to the basis $\mathcal{B}$ as before, according to eigenspaces and nilpotencey degrees and we write $\x=(\x_{\alpha, \ell})$. Note that with respect to this basis $\mathcal{B}$ the matrix $M$ is in a permuted Jordan form.
%
To bring our notation closer in line with the notation in \cite{Dy} we assign a number to each pair $(\alpha, \ell)$
according to the prescribed order, and we write instead $((\x_{\alpha,\ell}))=(\x_1, \ldots, \x_r)$. We know by Property 1 of Proposition \ref{myfoliationlemma}
that any $QSim_{D_M}$ map $G$ has the triangular form
$$G(\x_1, \ldots , \x_r)= (g_1(\x_1, \ldots, \x_r), \ldots, g_r(\x_r)).$$
Note that by Property 2 of Proposition \ref{myfoliationlemma} each map $g_i(\x_i, \ldots, \x_r)$ is bilipschitz in $\x_i$.
Furthermore, if we restrict to the map
$$G_i(\x_i, \ldots , \x_r) :=  (g_i(\x_i, \ldots, \x_r), \ldots, g_r(\x_r))$$
then $G_i$ is a $QSim_{D_{M_i}}$ map where $M_i$ is the submatrix of $M$ corresponding to the $(\x_i, \ldots, \x_r)$ coordinates.
This allows us to set up an induction argument in the same way as is done in Section 3 of
\cite{Dy}.

The base case is to consider  $$G_r(\x_r)=(g_r(\x_r)).$$
This action of $\G$ is a uniform quasisimilarity action on $V/(\oplus_{i=1}^{r-1} V_i)  \simeq V_r$. Therefore we can conjugate this action
to an action by similarities (see \cite{Dy}). To set up the induction step we write our map as $$G(\x,\y)=(g_\y(\x), g(\y))$$
where we assume that $g(\y)$ is an $ASim_{D_{M'}}$ for the appropriate $M'$ and $g_\y$ is bilipschitz in $\x$.
In fact this is all that is needed to prove the following proposition which is really the first part of the induction step of Theorem 2
in \cite{Dy} (Theorem \ref{tukia} here.)
\begin{prop}\label{simprop}
Let $\G$ be a group of $QSim_{D_M}$ maps that have the form
$$G(\x,\y)=(g_\y(\x), g(\y))$$
where $g(\y)$ is an $ASim_{D_{M'}}$ for the appropriate $M'$ and that $g_\y(\x)$ is bilipschitz in $\x$.
Then there exists a $QSim_{D_M}$ map $F$ such that each
$$\tilde{G} \in F \G F^{-1}$$
acts by $QSim_{D_M}$ maps of the form
$$\tilde{G}(\x, \y)= (\tilde{g}_\y(\x), \tilde{g}(\y))$$
where $\tilde{g}_\y$ is now a similarity as a function of $\x$ and $\tilde{g}(\y)$ is still an $ASim_{D_M}$ map.
\end{prop}
The proof of this proposition follows Section 3 up to Section 3.5 in \cite{Dy}.  The only observation that needs to be made is that
the standard dilitation $\delta_t(\x)$ which in \cite{Dy} was $ \delta_t(\x)=((t^{\alpha_i} \x_i)_{1\leq i \leq r})$ but now is $\delta_t(\x)=M^{\ln{t}}(\x)$. (Note that the two definitions coincide if $M$ is diagonal.)\\

\noindent{\bf Uniform constants.} After Proposition \ref{simprop} we have an action of $\G$ by maps of the form
$$ G(\x,\y)= (\lambda_\y A_\y(\x+B_\y), g(\y))$$
where $\lambda_\y\in \R$ and $A_\y$ is a rotation matrix.
We need to remove the dependance on $\y$ of $\lambda$ and $A$.  Again, as long as we interpret $\delta_t$ as $M^{\ln{t}}$ we can use the same proof for uniform multiplicative constant as in \cite{Dy}. The proof for uniform rotation constant, however, needs to be modified slightly, so we present it here.\\

\noindent{}{\bf Uniform rotation constant.}\label{rotation} At this point, we have a group $\G$ where each element has the form
\begin{equation}\label{By}
G(\x,\y)=M^s(g_\y(\x),g(\y))
	            =M^s(A_{\y}(\x +B_{\y}),g(\y))
\end{equation} 	
where $A_{\y} \in O(n)$, and $g(\y)$ is an $ASim_{D_{M'}}$ map, and $s \in
\mathbb{R}$ dependson the group element $G \in \G$.  The goal of the following proposition is
to show that $A_{\y}$ does not depend on $\y$. We do this by showing that if $A_{\y}
\neq A_{\y'}$ then $G$ is not a $K$-$QSim_{{D}_M}$ map.

\begin{prop}\label{rotconstprop} Let $G(\x,\y)=M^s( A_{\y}(\x +B_{\y}), g(\y) )$
be a $K$-$QSim_{{D}_M}$ map as above. Then $A_{\y}=A_{\y'}$ for all $\y,\y' \in \oplus_{j=i+1}^r V_j$.
\end{prop}

\begin{proof}
Suppose that for some $\y,\y'$ we have $A_{\y} \neq A_{\y'}$.  Then we can pick a
sequence $\z_{i} \in \oplus_{j=1}^r V_j$ such that
$$|A_{\y} \z_{i}  -A_{\y'} \z_{i} | > n_{i}$$ for a sequence of $n_{i} \rightarrow \infty$.
Next, for each $i$, pick $\x_{i},\x_{i}'$ such that  $\x_{i}+B_{\y}=\z_{i}$ and
$\x_{i}'+B_{\y'}=\z_{i}$.
Note that $$\x_i - \x_i'=B_{\y'}-B_{\y}.$$ Specifically, the distance between $\x_i$ and $\x_i'$ depends only on $\y$ and $\y'$ and not on $i$.
On the other hand, from the definition of $D_M$ we have that
$$ |A_{\y} \z_{i}  -A_{\y'} \z_{i}  + \xi(\y,\y')| \leq D_M (M^{-s}G(\x_{i},\y), M^{-s}G(\x_{i}',\y'))\leq e^{-s}  D_M (G(\x_{i},\y),G(\x_{i}',\y'))$$
where $\xi(\y,\y')$ is a function that does not depend on $i$.
Therefore
\bea |A_{\y} \z_{i}  -A_{\y'} \z_{i} |  &\leq& e^{-s} D_M (G(\x_{i},\y),G(\x_{i}',\y')) + \xi(\y,\y')\\
& \leq & K{D}_M((\x_{i},\y),(\x_{i}',\y')) + \xi(\y,\y')
\eea
But this is impossible since ${D}_M((\x_{i},\y),(\x_{i}',\y'))$ depends only on $\y,\y'$ and on $|\x_i-\x_i'|$ which, as pointed out above, does not depend on $i$. 
\end{proof}

Technically, using the structure of $M$ we should be able to show more here (in the non diagonal case certain directions are tied in with each other) but this will be unnecessary for our main application so we will not include it.
%
%
%
%
\section{Application: Quasi-isometric Rigidity.}\label{endgame}
In this section we consider $G_M\simeq \R \ltimes_M \R^n$ where $M$ is a matrix with $\det{M}=1$ that can be conjugated to a matrix of the form
$$\bm M_l & \\ & M_u^{-1} \fm$$
such that both $M_l$ and $M_u$ have eigenvalues of norm greater than one. Using Theorem \ref{tukia} we prove the following theorem
\begin{thm} Let $\Gamma$ be a finitely generated group quasi-isometric to $G_M$. Then $\Gamma$ is virtually a lattice in $G_{M^\alpha}$ for some $\alpha \in \R$.
\end{thm}
The proof outline follows the diagonal case exactly. By [EFW] and [P], $\Gamma$ quasi-acts by height-respecting quasi-isometries on $G_M$.
Therefore $\Gamma$ acts by $QSim_{D_{M_l}}$ maps on $\partial_l G_M \simeq \partial G_{M_l}$ and by $QSim_{D_{M_u}}$ maps on $\partial_u G_M \simeq \partial G_{M_u}$. Following Section 4.1 in \cite{Dy} these actions can be conjugated to an action of $\Gamma$ on $G_M$ by \emph{almost isometries} (see \cite{Dy} Section 4.1 for more details.) This is where Theorem \ref{tukia} is used. In Section 4.2 of \cite{Dy} the problem is reduced to studying a proper quasi-action of a subgroup $\N \subseteq \Gamma$ on $\R^n$ by $Bilip_{D_M}$ maps of the form
$$(\x_1, \ldots,  \x_r) \mapsto (\x_1 + B_1({\x_2, \ldots, \x_r}), \x_2 + B_2(\x_3, \ldots,\x_r), \ldots, \x_r +  B_r).$$
This reduction is the same in the nondiagonal case.

The next step is to show that any group acting in such a manner must be finitely generated polycyclic. To prove this, \cite{Dy} uses bounds on the $B_i$'s and induction to building a finite generating set for $\N$. In the diagonal case the bounds on the $B_i$'s are simply H\"older bounds but in the nondiagonal case the bounds are more complicated.
The following lemma provides the bounds in the non-diagonal case.
\begin{lemma}\label{previouslemma} Let $\y=(\x_{i+1},\cdots, \x_r)$ and $\y'=(\x_{i+1}',\cdots, \x_r')$. Then
$$|B_i(\y) - B_i(\y')| \leq \Xi_i( |\y-\y'|) $$
where $\Xi_i$ is some function with the property that $\Xi_i(w) \to 0$ as $w \to 0$.
\end{lemma}
\begin{proof}
Suppose $\gamma \in \N$. Set $p=(\x,\y), q=(\x,\y')$ and note that the $i$th coordinate of $\gamma p - \gamma q$ is $B_i(\y)-B_i(\y')$.
The result follows by applying Lemma \ref{XitheMap}.
\end{proof}\\
Now, as in Lemma 22 from \cite{Dy}, we derive even stronger bounds using the group structure of $\N$.
\begin{lemma}\label{Irineslemma}
  If $\g \in \N$ then 
$$|B_{i,\g}(\y)-B_{i,\g}(\y')| \leq \e_{i,\g}$$
where $\e_{i,\g}$ is a bound that depends only on the bilipschitz constant $K$ and on the bounds for the functions $B_{j,\g}$ for $j>i$.


\end{lemma}\begin{proof}
We will work by induction. The case $i = r$ is clear since $B_{r,\g}$
is a constant. We now assume the above statements hold for $j>i$ and prove it for $j=i$.
In other words, we assume that for $j>i$ there is some constant $B^{max}_{j,\g}$ such that $B_{j,\g}(\y) \leq B^{max}_{j,\g}$ for all $\y$.
 To aide with notation let $\x=(\x_1, \cdots, \x_i) $. Then we can write
 $$ (\x_1, \cdots, \x_r)=( \x, \y).$$
  For $\g
  \in \N$ we have
$$\g(\x,\y)=(\cdots , \x_i + B_{i,\g}(\y), \cdots)$$
and
$$ \g^n(\x,\y)=(\cdots, \x_i+ B_{i,\g^n}(\y), \cdots).$$
By an abuse of notation we will also write
$$\g \y:= (\x_{i+1} +B_{i+1,\g}(\x_{i+2}, \ldots, \x_r), \cdots, \x_r +  B_{r,\g}).$$
First notice that
\bea
	D_M( (\x,\y) ,(\x, \g \y)) &\leq& \max_{j>i}\{ |\sum \frac{t^k}{k!} B_{j,\g}(\y)| \} =: \chi(B_{j,\g}(\y))\\
						   &\leq& \max_{j>i}\{ |\sum \frac{t^k}{k!} B^{max}_{j,\g})|\} =: \chi(B_{j,\g}^{max}).
\eea
The second inequality follows by induction. Note that since the $B_{j,\g}(\y)$ are bounded then so is $t$ since
$t=\ln ( D_M((\x,\y),(\x,\g \y)))$ is a function of the $B_{j,\g}(\y)$. The exact bound is unnecessary. All we need to know is that $\chi$ is a function with $\chi(\omega) \to 0$ as $ \omega \to 0$. We will use $\chi$ to represent any function with these properties.
Next we list two observations that will be useful in the calculation that follows:
\begin{equation}\label{iterateBeqn}
B_{i,\g^n}(\y)=B_{i,\g}(\y)+B_{i,\g}(\g \y)\cdots
B_{i,\g}(\g^{n-1}\y),
\end{equation}
\begin{equation}\label{bounditerateDM}
D_M(\g^n(\x,\y),\g^n(\x,\g \y)) \leq K D_M( (\x,\y) ,(\x, \g \y)) \leq K \chi(B_{j,\g}^{max}).
\end{equation}
Equation \ref{iterateBeqn} allows us to write
$$|B_{i,\g}(\y)-B_{i,\g}(\g^n \y)| =|B_{i,\g^n}(\y)-B_{i,\g^n}(\g \y)| $$
and combined with Equation \ref{bounditerateDM} we have
$$|B_{i,\g}(\y)-B_{i,\g}(\g^n \y)+ \sum_{k=1}^{r-i} \frac{t^k}{k!} (B_{i+k,\g}(\y)-B_{i+k,\g}(\g^n \y))| \leq K^{\alpha_i} D_M( (\x,\y) ,(\x, \g \y))^{\alpha_i}.$$
In this case  $t=\ln ( D_M(\g^n(\x,\y),\g^n(\x,\g \y))) \leq  \ln ( K \chi(B_{j,\g}^{max}) )$ by Equation \ref{bounditerateDM}.
This allows us to derive the following estimate:
\bea
|B_{i,\g}(\y)-B_{i,\g}(\g^n \y)| &\leq& K^{\alpha_i} D_M( (\x,\y) ,(\x, \g \y))^{\alpha_i} + \sum \frac{t^k}{k!} |B_{i+k,\g}(\y)-B_{i+k,\g}(\g^n \y)|\\
							        &\leq& K^{\alpha_i} \chi(B_{j,\g}^{max})^{\alpha_i} + \sum \frac{(\ln ( K \chi(B_{j,\g}^{max}))^k}{k!} |B_{i+k,\g}(\y)-B_{i+k,\g}(\g^n \y)|\\
								 & \leq& K^{\alpha_i}   \chi(B_{j,\g}^{max})^{\alpha_i}  + \sum \frac{(\ln ( K \chi(B_{j,\g}^{max}))^k}{k!} |B_{i+k,\g}^{max}|\\
								 &\leq & \chi'( B_{j,\g}^{max}).
 \eea
Again, $\chi'$ depends only on $B_{j,\g}^{max}$ for $j>i$.\\

\noindent{}Now by Lemma \ref{previouslemma} we know that for arbitrary $\y$ and $\y'$
\bea|B_{i,\g}(\y)+B_{i,\g}(\g \y) \cdots B_{i,\g}(\g^{n-1}\y)
	- B_{i,\g}(\y')  \cdots -B_{i,\g}(\g^{n-1} \y')|&=&|B_{i,\g^n}(\y)-B_{i,\g^n}(\y')|\\
	&\leq&\Xi_i(|\y-\y'|) .
\eea
We also know that for all $s$
$$|B_{i,\g}(\g^s \y)-B_{i,\g}(\y)|< \chi'( B_{j,\g}^{max})$$
and $$|B_{i,\g}(\g^s\y')-B_{i,\g}(\y')|< \chi'( B_{j,\g}^{max})$$ so that
$$|nB_{i,\g}(\y)-nB_{i,\g}(\y')| \leq \Xi_i(|\y-\y'|)+ 2n \chi'( B_{j,\g}^{max})$$
for all $n$. In particular
$$|B_{i,\g}(\y)-B_{i,\g}(\y')| \leq \frac{\Xi_i(|\y-\y'|)}{n}+ 2 \chi'( B_{j,\g}^{max})$$
so that as $n\to \infty$
$$|B_{i,\g}(\y)-B_{i,\g}(\y')|\leq 2 \chi'( B_{j,\g}^{max}).$$
\end{proof}\\
The above lemma is the only ingredient needed to show that $\Gamma$ is polycyclic.
Once we know that $\Gamma$ is polycyclic we know by work of Mostow \cite{Mos} that $\Gamma$ is (virtually) a lattice in some solvable Lie group $\mathcal{L}$. Section 4.3 in \cite{Dy} combines work of  \cite{Cor,F,Ge,Gu,O} to show that $\mathcal{L}\simeq \R \ltimes_{M'} \R^n$. Finally, the proof is finished by Theorem 5.11 from \cite{FM3} which concludes that $M'$ has the same Jordan form as $M^{\alpha}$ for some $\alpha\in \R$.

\end{document}